\documentclass{amsart}
\usepackage{a4wide}
\usepackage{amsmath, amssymb, epsfig, graphicx}
\usepackage[dvips]{rotating}
\usepackage{booktabs}
\usepackage{url}

\let\oldmarginpar\marginpar
\renewcommand\marginpar[1]{\-\oldmarginpar[\raggedleft\footnotesize #1]%
{\raggedright\footnotesize #1}}

\usepackage[colorlinks]{hyperref}
\hypersetup{linkcolor=blue, urlcolor=blue, citecolor=red}

\newtheorem{theorem}{Theorem}[section]
\newtheorem{defi}[theorem]{Definition}

\newtheorem{lemma}[theorem]{Lemma}
\newtheorem{corollary}[theorem]{Corollary}
\newtheorem{question}[theorem]{Question}

\newtheorem{rem}[theorem]{Remark}

\newtheorem{ex}[theorem]{Example}

\newtheorem{CSP}[theorem]{CSP}

\newcommand{\N}{\mathbb{N}}

\title{Finite nilpotent semigroups of small coclass}
\author{Andreas Distler}
\thanks{The author thanks Bettina Eick and James D.~Mitchell for
  comments on earlier versions of the paper. The work was developed
  with support from the doctoral program of the University of St
  Andrews and from the project PTDC/MAT/101993/2008 of Centro de
  \'Algebra da Universidade de Lisboa, financed by FCT and FEDER}

\begin{document}

\begin{abstract}
The parameter coclass has been used successfully in the study of
nilpotent algebraic objects of different kinds. In this paper a
definition of coclass for nilpotent semigroups is introduced and semigroups
of coclass 0, 1, and 2 are classified. Presentations for all such
semigroups and formulae for their numbers are obtained. The
classification is provided up to isomorphism as well as up to
isomorphism or anti-isomorphism. Commutative and self-dual semigroups
are identified within the classification.
\end{abstract}
\maketitle

\section{Introduction}
\noindent
Nilpotency is an important concept in many areas of algebra. For
semigroups there exist two common definitions; on the one hand the
generalisation of nilpotency for groups introduced by
Mal$'$cev~\cite{Mal53}, and on the other hand a 
natural adaptation of the notion for algebras. The latter is used in
this paper: a semigroup $S$ is \emph{nilpotent} if there exists $c
\in \N_0$ such that $|S^{c+1}|=1$, the least such $c$ is the
\emph{(nilpotency) class} of $S$.

A parameter used successfully in the studies of nilpotent groups and
Lie algebras is the coclass of such objects~\cite{LN80, SZ97}. For a
finite nilpotent semigroup $S$ of class $c$ we define the
\emph{(nilpotency) coclass}  of $S$ to be $|S|-1-c$. 

It is immediate from the definition that every nilpotent semigroup $S$
contains a zero and that $|S|-1$ is an upper bound for both class and
coclass. Throughout the paper attributing class or coclass to a
semigroup $S$ shall imply that $S$ is nilpotent.  

The above definition of class and coclass is in parallel to that for other
algebraic structures. It ensures for all $k\in \N$ that
the class of a nilpotent semigroup $S$ equals the sum of the classes
of the ideal $S^k$ and the quotient $S/S^k$.
Furthermore class and coclass of $S$ equal the respective attributes of
the nilpotent algebra naturally associated with $S$ over any given field,
that is the contracted semigroup algebra of $S$ (compare~\cite{DE12}).

The main results obtained in this paper are complete classifications of
nilpotent semigroups of coclass 1 and of coclass 2. These demonstrate that
coclass is a useful parameter for the classification, seemingly better
suited than the more directly defined nilpotency class.

The forthcoming section contains technical background on basic properties
of nilpotent semigroups. An immediate consequence is the description
of semigroups of coclass $0$ in Lemma~\ref{lem_nil_mono}. In
Sections~\ref{sec_cc1} and~\ref{sec_cc2} the main results of the paper
are obtained. Classifications of semigroups of coclass $1$ respectively $2$
are given in Theorem~\ref{thm_cc1} respectively Theorems~\ref{thm_cc2_d2}
and~\ref{thm_cc2_d3}. Presentations of the semigroups are provided in all
cases. Ideas and problems for an extension of the results to higher coclass
are briefly discussed at the end of Section~\ref{sec_cc2}.
In the final section the main results are applied to 
the enumeration of semigroups of coclass $1$ and $2$. Formulae for the
numbers of such semigroups are given along with a table containing
the numbers for small orders. 

The lists of presentations from the main theorems have been implemented in
\textsc{GAP}~\cite{GAP4} by the author. The implementation is available in the
package \textsc{Smallsemi}~\cite{smallsemi} and can be accessed through the
function
\texttt{PresentationsOfNilpotentSemigroups}\footnote{Currently under
  development}.

\section{Preliminaries and Coclass 0}
\label{sec_nil_rank}
\noindent
This section contains basic results about the structure of nilpotent
semigroups and a characterisation of semigroups with coclass $0$.

\begin{lemma}
\label{lem_part_nil}
Let $S$ be a nilpotent semigroup of class $c$. Then the following hold:
\begin{enumerate}
\item 
the sets $S^k\setminus S^{k+1}$ with $1 \leq k \leq c$ are non-empty
and form a partition of $S\setminus S^{c+1}$;
\item
if $s= s_1s_2\cdots s_k \in
S^k\setminus S^{k+1}$ with $1\leq k\leq c$, then $s_i\cdots
s_j \in S^{j-i+1}\setminus S^{j-i+2}$ for all $1\leq i \leq j \leq k$;
\item
if $|S^l \setminus S^{l+1}|=1$ for $1 \leq l \leq c$ then $|S^k
\setminus S^{k+1}|=1$ for all $l \leq k \leq c$.
\end{enumerate}
\end{lemma}
\begin{proof}
(i): For any three sets $A,B,C$ with $A\supseteq
B\supseteq C$ the set $A\setminus C$ equals the disjoint union of
$A\setminus B$ and $B\setminus C$. Hence it suffices to show that the
sets $S^{k}\setminus S^{k+1}$ are non-empty for $1\leq k \leq c$. From
$S^{k}= S^{k+1}$ it would follow that $S^{k}=S^{c+1}$ contains just
one element and $c=k-1$, a contradiction to $k\leq c$. Thus 
$S^{k+1}$ is a proper subset of $S^{k}$ for $1\leq k \leq c$.

(ii): The statement is shown by contradiction. Assume that $s_i\cdots
s_j\in S^{j-i+2}$ for some $i,j$ with $1\leq i \leq j \leq k$. This means
$s_i\cdots s_j$ can be expressed as a product $t_1\cdots t_{j-i+2} \in
S^{j-i+2}$. Replacing $s_i\cdots s_j$ by $t_1\cdots t_{j-i+2}$ in $s$,
that is
\[
s=s_1s_2\cdots s_k = s_1\cdots s_{i-1}t_1\cdots t_{j-i+2}s_{j+1}\cdots s_k,
\]
yields $s \in S^{k+1}$, a contradiction.

(iii): Let $s=s_1s_2\cdots s_k\in S^k\setminus S^{k+1}$ for some $l < k
\leq c$. According to Part (ii) the product
$s_1s_2\cdots s_l$ of length $l$ equals the unique element
$t_1t_2\cdots t_l$ in $S^l \setminus S^{l+1}$. Hence $s_1s_2\cdots s_k =
t_1t_2\cdots t_l s_{l+1}\cdots s_k$. The right hand side of this
equation is still a product of length $k$ equalling $s$. Thus, again by
Part (ii) $t_2t_3\cdots t_ls_{l+1}$ equals the unique element in $S^l
\setminus S^{l+1}$ and can be replaced by $t_1t_2\cdots t_l$. Applying
this argument repeatedly and always replacing the product of length
$l$ with $t_1t_2\cdots t_l$ yields $s=t_1^{k-l+1}t_2\cdots t_l$. This
implies $|S^k\setminus S^{k+1}|=1$, since $s \in S^k\setminus S^{k+1}$
was chosen arbitrarily.
\end{proof}

The previous lemma will be used repeatedly  and it is worthwhile to
gain some intuition for
its content. First of all, for each element other than the zero, the
length of a product equalling the element is restricted. An element in
$S^k \setminus S^{k+1}$ can be written as product of length $k$,
but not as product of length $k+1$. 
We obtain a partition of $S$ if we collect all elements with the same
maximal length of a product equalling the element in a separate part.
A semigroup of
class $c$ is partitioned into $c+1$ sets: one for each maximal length
between 1 and $c$ and the zero element in a set by itself. This yields
precisely the partition in Lemma~\ref{lem_part_nil}(i).
Further, each part of a product of maximal length is
clearly maximal itself, which is essentially what is stated in the
second part of Lemma~\ref{lem_part_nil}. It follows in particular 
that every element in a product of maximal length lies in $S\setminus S^2$.
That in addition each element in $S\setminus S^2$ clearly has to appear
in every generating set, yields the following well-known result. 

\begin{lemma}
\label{lem_gen}
Let $S$ be a nilpotent semigroup containing at least $2$ elements. Then
$S\setminus S^2$ is the unique minimal generating set of $S$.
\end{lemma}

We shall see that the previous lemma together with Lemma~\ref{lem_part_nil}(i)
yields that a semigroup with coclass $0$ can be generated by one 
element. A finite semigroup generated by a single element $u$ which satisfies
$u^k=u^{k+l}$ for minimal $k,l\in\N$ is called \emph{monogenic} of 
\emph{index} $k$ and \emph{period} $l$.

\begin{lemma}
\label{lem_nil_mono}
Let $S$ be a semigroup of order $n\in\N$.
Then the following statements are equivalent:
\begin{enumerate}
\item $S$ is nilpotent of coclass $0$;
\item $S$ is monogenic and nilpotent;
\item $S$ is monogenic with period $1$;
\item $\langle u\mid u^n=u^{n+1}\rangle$ is a presentation for $S$.
\end{enumerate}
\end{lemma}
\begin{proof}
We may assume $n\geq 2$ as the statement is obvious for $n=1$.

(i) $\Rightarrow$ (ii):
According to Lemma~\ref{lem_part_nil}(i) the sets $S^k
\setminus S^{k+1}$ with $1\leq k \leq n-1$ are non-empty. Hence each set
contains exactly one element. The set $S\setminus S^2$ is a generating
set by Lemma~\ref{lem_gen} and thus $S$ is monogenic.

(ii) $\Rightarrow$ (iii):
If $u$ denotes the generator of $S$, then $u^n$ and $u^{n+1}$
both equal the zero element. Therefore the period of $S$ is
$1$.

(iii) $\Rightarrow$ (iv):
If $u$ denotes the generator of $S$ all of $u,u^2,\dots, u^n$ have 
to be pairwise different. And that the period of $S$ is $1$ implies
that equality $u^n=u^{n+1}$ holds.

(iv) $\Rightarrow$ (i):
Clearly $u^n$ is a zero and every product
of $n$ elements equals $u^n$, making $S$ nilpotent. Moreover,
$S^{k}\setminus S^{k+1} = \{u^{k}\}$ for $1\leq k \leq n-1$, showing
that $S$ has nilpotency coclass $0$.
\end{proof}

In a nilpotent semigroup $S$ every element generates a nilpotent
monogenic subsemigroup of coclass $0$. The class of
each such subsemigroup is at most the class of $S$. Of special
interest is the case when equality of the classes holds for some
elements in $S$.
\begin{lemma}
\label{lem_struc_nil}
Let $S$ be a nilpotent semigroup of class $c$. If $|S^{c-1}\setminus
S^{c}|=1$ then there exists an element in $S$ that generates a
nilpotent subsemigroup of class $c$.
\end{lemma}
\begin{proof}
For $s \in S^{c}\setminus S^{c+1}$ take a product $s_1s_2\cdots s_{c}$
equal to $s$. Then both $s_1\cdots s_{c-1}$ and $s_2\cdots s_{c}$
are in $S^{c-1}\setminus S^{c}$ due to Lemma~\ref{lem_part_nil}(ii) and
hence $s_1\cdots s_{c-1}$ equals $s_2\cdots s_{c}$. This yields
$s = s_1s_2s_3\cdots s_{c} = s_1s_1s_2\cdots s_{c-1}$.
Repeating the process starting with $s=s_1s_1s_2\cdots s_{c-1}$ gives
$s=s_1s_1s_1s_2\cdots s_{c-2}$ and leads after $c-1$ iterations to
$s=s_1^{c}$. Conclude by Lemma~\ref{lem_part_nil}(ii) that $s_i^k$
for $1\leq k\leq c$ are pairwise different and are also different from
the zero $s_1^{c+1}$. Hence the semigroup generated by $s_1$ has size
$c+1$ and class $c$.
\end{proof}

The condition in the previous lemma is rather technical. For a fixed
coclass we can turn it into a restriction on the size of the
semigroup. The restriction can be strengthened if also the size of the
generating set is incorporated.

\begin{corollary}
\label{coro_bound}
Let $S$ be a finite, nilpotent semigroup of class $c$ and coclass
$r$. If either $c \geq r+2$ or $c\geq r+4-|S\setminus S^2|$ then there
exists an element in $S$ that generates a nilpotent subsemigroup of
class $c$.
\end{corollary}
\begin{proof}
By Lemma~\ref{lem_part_nil}(iii) it
is immediate that either of the two conditions in the statement
implies that $S$ satisfies the prerequisite in
Lemma~\ref{lem_struc_nil}.
\end{proof}

Note that for a nilpotent semigroup $S$ the only case in which $|S\setminus S^2|=1$
is, according to Lemma~\ref{lem_nil_mono}, if $S$ has coclass $0$. It
follows that the generating set for a semigroup of positive coclass
contains at least two elements and that the second bound in the
previous corollary is at least as good as the first bound, usually better.

\section{Classification for Coclass 1}
\label{sec_cc1}
\noindent
Some general remarks on the classification of semigroups are in place
before stating the results of this section. Two semigroups are
\emph{anti-isomorphic} if one is isomorphic to the dual of the other,
and a semigroup is \emph{self-dual} if it is isomorphic to its
dual. 
Semigroups are usually classified up to isomorphism or
anti-isomorphism, and the results here are stated accordingly. For the
sake of brevity we denote `isomorphism or anti-isomorphism' by
\emph{(anti-)isomorphism} and correspondingly use
\emph{(anti-)isomorphic} to mean `isomorphic or anti-isomorphic'. 
For computational purposes it is often useful to
work with a classification just up to isomorphism. In all results the
self-dual semigroups are indicated thus allowing the reader to extract
each classification up to isomorphism.
  
We now use the structural information provided in the previous
section to classify nilpotent semigroups of coclass 1. Though instead of 
doing so directly we prove a generalisation to arbitrary coclass 
which will also be useful for the classification of semigroups of
coclass 2 in the forthcoming section.

\begin{lemma}
\label{lem_coclass_d}
For $c,r\in\N$ with $c\geq 3$ the following is a complete list up to
(anti-)isomorphism of representatives of nilpotent semigroups of class
$c$ and coclass $r$, in which the unique minimal generating set is of
size $r+1$ and that contain at least $r$ copies of the monogenic,
nilpotent semigroup of class $c$:
\begin{eqnarray*}
\mathcal{H}_k & \!\!\!\! = \langle u_1,u_2,\dots,u_r,v \mid &\!\!\!\! u_1^{c+1}=u_1^{c+2};
 u_1^2=u_i^2=u_iu_j=u_ju_i, 1 \leq j < i \leq r;\\
 &&\!\!\!\! u_iv=vu_i=u_1^k, 1 \leq i \leq r; v^2=u_1^{2k-2}\rangle,
2\leq k\leq c-1;\\
\mathcal{J}_k &\!\!\!\! = \langle u_1,u_2,\dots,u_r,v \mid &\!\!\!\! u_1^{c+1}=u_1^{c+2};
  u_1^2=u_i^2=u_iu_j=u_ju_i, 1 \leq j < i \leq r;\\
 &&\!\!\!\! u_iv=vu_i=u_1^k, 1 \leq i \leq r; v^2=u_1^{c}\rangle,
\lfloor c/2\rfloor+2 \leq k\leq c-1;\\
\mathcal{X} &\!\!\!\! = \langle u_1,u_2,\dots,u_r,v \mid &\!\!\!\! u_1^{c+1}=u_1^{c+2};
  u_1^2=u_i^2=u_iu_j=u_ju_i, 1 \leq j < i \leq r;\\
 &&\!\!\!\! u_iv=vu_i=u_1^{(c+2)/2}, 1 \leq i \leq r; v^2=u_1^{c+1}\rangle, 
\mbox{ if } c\equiv 0\mod 2;\\
\mathcal{N}_{k,l,m}^e &\!\!\!\! =\langle u_1,u_2,\dots,u_r,v \mid & \!\!\!\! 
u_1^{c+1}=u_1^{c+2}; u_1^2=u_i^2=u_iu_j=u_ju_i, 1 \leq j < i \leq r;\\
  &&\!\!\!\! vu_i=u_iv=u_1^{c+1}, 1\leq i \leq k;\\
  &&\!\!\!\! vu_i=u_1^{c+1}, u_iv=u_1^{c}, k+1\leq i\leq l;\\
  &&\!\!\!\! vu_i=u_1^{c}, u_iv=u_1^{c+1}, l+1\leq i\leq m;\\
  &&\!\!\!\! vu_i=u_iv=u_1^{c}, m+1\leq i \leq r; v^2=u_1^{c+e}\rangle,\\
  &&\!\!\!\! 0\leq k \leq m\leq r, k\leq l \leq \lfloor(k+m)/2\rfloor, e\in \{0,1\}.
\end{eqnarray*}
All semigroups in the list, except $\mathcal{N}^e_{k,l,m}$ with 
$l\neq (k+m)/2$, are self-dual.
\end{lemma}

\begin{proof}
Let $S=\langle u_1,u_2,\dots,u_r,v \rangle$ be a nilpotent semigroup
fulfilling the conditions described in the statement of the lemma where
$u_1,u_2,\dots, u_r$ denote $r$ elements of $S$ that each generate a
subsemigroup of class $c$. According to Lemma~\ref{lem_part_nil}(i) the
sets $S^k\setminus S^{k+1}$ for $2\leq k\leq c$ each contain exactly one
element which consequently equals $u_i^k$ for all $1\leq i \leq r$. As
the index $i$ does not influence the value of $u_i^k$ for $k\geq 2$ we
simplify notation and write just $u^k$.

For given $1\leq i,j \leq r$ the equality $u_iu_j=u^k$ holds for some
$2\leq k \leq c+1$. It follows $u_i(u_iu_j) = u_iu^k= u^{k+1}$
and also $(u_iu_i)u_j = u^2u_j= u^3$. This proves $k=2$ using
that $u^3$ is not the zero because $3 < c+1$ by assumption. As $i,j$
were arbitrary $u_iu_j$ equals $u^2$ for all $1\leq i,j\leq r$.

We shall now conduct a case distinction depending on the value of
$u_1v$. Note that the whole of $S$ is determined if we know the value
of $v^2$ and the values of $u_jv$ and $vu_j$ for all $1\leq j\leq r$ because
$u^kv$ and $vu^k$ can then be deduced for all $2\leq k \leq c+1$.
The choices not contradicting associativity are determined below.
Choices leading to (anti-)isomorphic semigroups and are identified; 
no two semigroups from
different cases are (anti-)isomorphic. Note that any (anti-)isomorphism 
sends generators to generators and is as such induced by a permutation
of $u_1,u_2,\dots,u_r,v$.

{\bf Case 1:} $u_1v = u^k$ with $2\leq k \leq \lceil c/2\rceil$. 
Let $l \in \{2,3,\dots,c+1\}$ such that $vu_1=u^l$. From $u^{k+1}= u_1vu_1
= u^{l+1}$ it follows that $k=l$ as $u^{k+1}$ is not the zero. Using the same
type of argument considering $u_jvu_1 = u_ju^k=u^{k+1}$ it follows $u_jv=u^k$
and similarly $vu_j=u^k$ for all $2\leq j\leq r$. Also $v^2=u^m$ implies
$u^{m+1}=v^2u=vu^k=u^{2k-1}$, and hence $m=2k-2$ as $u^{2k-1}$ is not
the zero. 
It follows that the value of any proper product of generators is
determined by its length and by how many times $v$ appears (A
product of length $i+j$ that contains $j$ times $v$ equals
$u^{i+j(k-1)}$.). This guarantees that the multiplication does not
contradict associative. Hence $S$ is a homomorphic image of $\mathcal{H}_k$.
Apart from the generators each presentation contains at most the
elements $u^k$ for $2\leq k \leq c+1$ and has therefore size $c+2$ and
class $c$. For all values of $k$ the semigroup is commutative and no
two are (anti-)isomorphic since $v^2$ is different for different $k$. (Note
that if $k=2$, than $\langle v\rangle$ is also of class $c$ and any
permutation of the generators induces an automorphism of $S$.)

{\bf Case 2:} $u_1v = u^k$ with $\lceil c/2\rceil < k \leq
c-1$. As in the previous case $vu_j=u_jv=u^k$ for all $1\leq j\leq
r$. Now $vvu=u^{2k-1}$ equals the zero $u^{c+1}$. This leaves the two
choices $u^{c}$ and $u^{c+1}$ for $v^2$. Again, the value of a
product of generators only depends on its length and the number of
times $v$ appears, making the multiplication associative. Similar to
Case 1 it follows that $S$ is isomorphic to one of the presentations 
$\mathcal{H}_k$ or $\mathcal{J}_k$ respectively $\mathcal{X}$ if $k=(c+2)/2$ 
depending on the value of $v^2$. All these presentations define pairwise
not (anti-)isomorphic, commutative semigroups.

{\bf Case 3:} $u_1v \in \{u^{c},u^{c+1}\}$.
Let $l \in \{2,3,\dots,c+1\}$ such that $vu_1=u^l$. From
$u_1(vu_1)=u^{l+1}$ and the fact that $(u_1v)u_1$ equals the zero, $u^{c+1}$,
it follows that $l\in \{c,c+1\}$. Considering $u_jvu_1$ and
$u_1vu_j$ one then shows with the same type of argument that $u_jv,
vu_j \in \{u^{c},u^{c+1}\}$ for all $2\leq j\leq r$. In a similar
way $v^2=u^m$ leads to $u^{m+1}=v(vu)=vu^l=u^{2l-1}=u^{c+1}$, and hence $m
\in\{c,c+1\}$. Every choice gives an associative multiplication as
all products of three elements involving $v$ equal $u^{c+1}$.
Some multiplications lead to (anti-)isomorphic semigroups because every
permutation of $u_1,u_2,\dots,u_r$ yields an isomorphism and an
anti-isomorphism. To find representatives we partition the
set $\{u_j\mid 1\leq j \leq r\}$ into four parts depending on $u_jv$
and $vu_j$ equalling $u^c$ or $u^{c+1}$. Two semigroups from this case are
isomorphic if and only if they have the same number of generators of
each of the four types and $v^2$ takes the same value; and they are
anti-isomorphic if by interchanging the two types with $vu_j\neq u_jv$ in
one of the semigroups we obtain two isomorphic semigroups. Hence a semigroup
from this case is self-dual if there is the same number of generators
of the two types with $vu_j\neq u_jv$. Presentations for every possible choice
for the numbers of types and the value of $v^2$ up to (anti-)isomorphism are
then given by $\mathcal{N}_{k,l,m}^e$ as in the statement of the lemma.
\end{proof}

For nilpotent semigroups of class $2$ the step in the proof of
Lemma~\ref{lem_coclass_d} that restricts the possible results for
products of two generators does not work. Indeed, for semigroups of
class $2$ every combination can occur, a fact that is used
in~\cite{DM12} to count such semigroups of any order.

\begin{theorem}
\label{thm_cc1}
For $n\in\N$ with $n\geq 5$ the following is a complete list up to
(anti-)isomorphism of representatives of nilpotent semigroups of order
$n$ and coclass $1$:
\begin{eqnarray*}
&H_k &\!\!\!\!= \langle u,v \mid u^{n-1}=u^n, uv=u^k, vu=u^k,v^2=u^{2k-2}
\rangle, 2\leq k\leq n-1;\\
&J_k &\!\!\!\!= \langle u,v \mid u^{n-1}=u^n, uv=u^k, vu=u^k,v^2=u^{n-2}
\rangle, n/2< k\leq n-1;\\
&X &\!\!\!\!= \langle u,v \mid u^{n-1}=u^n, uv=u^{n/2},
vu=u^{n/2},v^2=u^{n-1}\rangle, \mbox{ if } n\equiv 0 \mod 2;\\
&N_1 &\!\!\!\!= \langle u,v \mid u^{n-1}=u^n, uv=u^{n-1},
vu=u^{n-2},v^2=u^{n-2} \rangle;\\
&N_2 &\!\!\!\!= \langle u,v \mid u^{n-1}=u^n, uv=u^{n-1},
vu=u^{n-2},v^2=u^{n-1} \rangle.
\end{eqnarray*}
All semigroups in the list, except $N_1$ and $N_2$, are self-dual.
\end{theorem}

\begin{proof}
From Lemmas~\ref{lem_part_nil} and~\ref{lem_gen} it follows that every
semigroup of coclass $1$ is generated by $2$ elements. Together with
Corollary~\ref{coro_bound} this implies that Lemma~\ref{lem_coclass_d}
covers all semigroups of coclass $1$ and order at least $5$. The presentations
in the statement are obtained by choosing $r=1$ in Lemma~\ref{lem_coclass_d}.
(There is a slight discrepancy in nomenclature as the four semigroups
$H_{n-2}, H_{n-1}, J_{n-2},$ and $J_{n-1}$ correspond to $\mathcal{N}_{0,0,0}^1,
\mathcal{N}_{1,1,1}^1,\mathcal{N}_{0,0,0}^0,$ and $\mathcal{N}_{1,1,1}^0$.)
\end{proof}

Note that there are no semigroups of coclass $1$ and order $1$ or $2$,
and there is only the zero semigroup of order $3$. For $n=4$ the
presentations given in Theorem~\ref{thm_cc1} yield semigroups of order
$4$ and coclass $1$, but $N_1$ becomes self-dual and the following two
presentations of self-dual semigroups are missing from a complete list:
\[
\langle u,v \mid u^2=u^3, u^2=v^2, uv=vu, u^2=v^3 \rangle \mbox{
  \ and \ }\langle u,v \mid u^2=u^3, u^2=v^2, u^2=uv, u^2=v^3 \rangle.
\]

\section{Classification for Coclass 2}
\label{sec_cc2}
\noindent
For some general remarks on the classification see the beginning of
Section~\ref{sec_cc1}. The natural next step is to consider nilpotent
semigroups of coclass~$2$. Due to Lemmas~\ref{lem_part_nil}(i)
and~\ref{lem_nil_mono} the minimal generating sets of such semigroups
contain either $2$ or $3$ elements. These two possibilities shall be
treated separately. 

\begin{theorem}
\label{thm_cc2_d2}
For $n\in\N$ with $n\geq 7$ the following is a complete list up to 
(anti-)isomorphism of representatives of nilpotent semigroups of order
$n$ and coclass $2$ whose minimal generating set has size $2$:
\begin{align*}
&T_{1,i} =\langle u,v \mid u^{n-2}=u^{n-1}, uv = vu,
v^2 = uv, v^3 = u^{n-i}\rangle, i\in\{2,3\};\\
&T_{2,k} =\langle u,v \mid u^{n-2}=u^{n-1}, uv = vu, v^2 =
u^{2k-4}, u^2v = u^k \rangle, 3 \leq k < n/2;\\
&T_{2,i,k} =\langle u,v \mid  u^{n-2}=u^{n-1}, uv = vu, v^2 =
u^{n-i}, u^2v = u^k \rangle,   n/2 \leq k \leq n-2,i \in\{2,3,4\};\\
&T_{3} =\langle u,v \mid  u^{n-2}=u^{n-1}, v^2 = uv, vu = u^2,
uv^2 = u^3 \rangle;\\
&T_{3,i} =\langle u,v \mid  u^{n-2}=u^{n-1}, v^2 = uv, vu =
u^{n-i}, uv^2 = u^{n-2} \rangle,i \in\{2,3\};\\
&T_{4,k} =\langle u,v \mid  u^{n-2}=u^{n-1}, uv = vu, uv = u^k, v^3 = u^{3k-3}
\rangle, 2 \leq k < n/3;\\
&T_{4,i,k} =\langle u,v \mid  u^{n-2}=u^{n-1}, uv = vu, uv = u^k, v^3 =
u^{n-i}\rangle, n/3 \leq k \leq n-4, i\in\{2,3\};\\
&T_{4,i,j,k} =\langle u,v \mid  u^{n-2}=u^{n-1}, uv = u^{n-i}, vu =
u^{n-j}, v^3 = u^{n-k}\rangle, i,j,k\in\{2,3\};\\
&T_{5,k} =\langle u,v \mid  u^{n-2}=u^{n-1}, uv = u^k, v^2 = u^{2k-2}, vu^2 = u^{k+1}
\rangle, 2 \leq k < n/2;\\
&T_{5,i,k} =\langle u,v \mid  u^{n-2}=u^{n-1}, uv = u^k, v^2 = u^{n-i}, vu^2 = u^{k+1}
\rangle, n/2 \leq k < n-5, i\in\{2,3\};\\
&T_{5,i,j,k} =\langle u,v \mid  u^{n-2}=u^{n-1}, uv = u^{n-i}, v^2 =
u^{n-j}, vu^2 = u^{n-k} \rangle, i\in\{2,3,4\}, j,k\in\{2,3\}.
\end{align*}
A semigroup from the list is self-dual if it is commutative or is $T_{5,2}$.
\end{theorem}
\begin{proof}
The proof uses the same methods as those of Lemma~\ref{lem_coclass_d}
and Theorem~\ref{thm_cc2_d3}. Let $S=\langle u,v\rangle$ be a
semigroup of order $n$ and coclass $2$. By Corollary~\ref{coro_bound}
we may assume without loss of
generality that $|\langle u\rangle|=n-2$. Further denote by $y$
the element in $S$, which is neither $v$ nor a power of $u$. As
$y\in S^2\setminus S^3$ at least one of $uv,vu,v^2$ has to equal
$y$. Note that the values of these products with two generators
together with the values of $vu^2,vy,uy$ uniquely define
$S$.\footnote{The remaining products are deduced as follows:
  $vu^k=(vu^2)u^{k-2}$; $u^kv=u^{k-1}(uv)$;
\begin{equation*} yu^k=\begin{cases}
vuu^k=(vu^2)u^{k-1} & \textrm{ if } vu=y\\
v^2u^k=v(vu)u^{k-1}=(vu)u^{l+k+2}=u^{2l+k-2} & \textrm{ if } vu=u^l, v^2=y\\
uvu^k = u(vu)u^{k-1}=u^{l+k}& \textrm{ if } vu=u^l,uv=y;
\end{cases}\end{equation*}
\begin{equation*} yv=\begin{cases}
v^2v=vv^2=vy & \textrm{ if } v^2=y\\
uvv = uu^l =u^{l+1} & \textrm{ if } v^2=u^l, uv=y\\
vuv = vu^k=(vu^2)u^{k-2} & \textrm{ if } uv=u^k, vu=y
\end{cases}\end{equation*}.
}
We shall determine the choices of values not contradicting
associativity in several cases depending on which products equal $y$.
For all resulting multiplications the value of a product with three
elements will only depend on the number of times $v$ and $y$ appear.
Such multiplications are obviously associative.

Each of the presentations from the list in the statement of the
theorem leads to a semigroup with at most $n$ elements. Hence,
whenever the considerations in the following imply that $S$ fulfils
the relations $R$ of one of the presentations then $\langle u,v\mid
R\rangle$ is a presentation for $S$. 

{\bf Case 1:} $y=v^2=uv=vu$. Let $vy=u^k$. Since $u$ and $v$ commute,
this yields $uy=uv^2=vuv=vy=u^k$ and $vu^2=uvu=uy=u^k$. Then 
\[
u^{k+1}=uvy=uyv=u^kv=u^{k-1}uv=u^{k-1}y=u^{k-2}u^k=u^{2k-2},
\]
which gives $k=3$ or $k\geq n-3$, and thus leads to 3
semigroups. For any $k\in\{3,n-3,n-2\}$ the respective
semigroup fulfils the relations of
\[
\langle u,v \mid u^{n-2}=u^{n-1}, uv = vu, v^2 = uv, v^3 = u^k\rangle.
\]

{\bf Case 2:} $y=vu=uv$. Let $v^2=u^l$. It follows
$vy=vvu=u^{l+1}$. Let $uy=u^k$ then $vu^2=uvu=uy=u^k$ and
\begin{equation}
\label{eq_nil_n-2}
u^{l+2}= v^2u^2=vuvu=vuy=vu^k=u^kv=u^{k-2}uy=u^{2k-2}.
\end{equation}

If $2\leq l \leq n-5$ then \eqref{eq_nil_n-2} gives $k=l/2+2$. That $k$
has to be an integer yields $\lceil n/2 \rceil - 3$
semigroups with presentations $T_{2,l}$. 

If $n-4\leq l \leq n-2$ then
$k\geq n/2$ due to \eqref{eq_nil_n-2}. Hence for each of the three
choices for $l$ there are $3(\lfloor n/2 \rfloor -1)$ semigroups with
presentation $T_{2,l,k}$.

{\bf Case 3:} $y=uv=v^2$. Let $vu=u^l$. Then
\[
u^{l+1}= uu^l=uvu=yu=vvu=vu^l=vuu^{l-1}=u^{2l-1}
\]
yields either $l=2$ or $l\geq n-3$. Furthermore $vu^2=u^{l+1}$ and
\[
uy=uv^2=uvv=yv=vvv=vy=vuv=u^lv=u^{l-2}uy,
\]
which shows that all three values for $l$ lead to valid choices for
$vu^2,uy,$ and $vy$ defining an associative multiplication. Hence this
case accounts for 3 semigroups with presentations $T_{3},T_{3,2}$ and $T_{3,3}$.

{\bf Case 3':} $y=vu=v^2$. This case leads to the dual semigroups to
those from Case~3.

{\bf Case 4:} $v^2=y$. Let $vu=u^k$ and $uv=u^l$. Then
$u^{k+1}=uvu=u^{l+1}$ and hence $k=l$ or $k,l \in
\{n-3,n-2\}$. Furthermore $vu^2=vuu=u^{k+1}$ and $uy =
uvv=u^lv=u^{2l-1}$. For the value of $vy$ consider
$uvy=u^lvv=u^{2l-1}v=u^{3l-2}$.

If $2 \leq l < n/3$ then $vy=u^{3l-3}$ and $l < n-3$, which
leads to $\lceil n/3\rceil -2$ semigroups with presentations $T_{4,l}$. 

If $n/3 \leq l \leq n-4$ then $vy \in
\{u^{n-3},u^{n-2}\}$, leading in this case to $2(n-4-\lceil n/3\rceil +1)$
semigroups with presentations $T_{4,2,l}$ and $T_{4,3,l}$ .

If $l\in\{n-3,n-2\}$ then again $vy \in \{u^{n-3},u^{n-2}\}$. Recall
that here $k\in \{n-3,n-2\}$. The two semigroups in which one of $k$ and $l$
equals $n-3$ and the other one $n-2$ are anti-isomorphic. Hence this case 
leads to $6$ semigroups with presentations $T_{4,a,b,c}$ with $a,b,c \in \{2,3\}$
and $a\leq b$.

{\bf Case 5:} $vu=y$. Let $uv=u^k$ and $v^2=u^l$. It follows
$vy=vvu=u^{l+1}$ and $uy=uvu=u^{k+1}$. For $vu^2$ consider $uvu^2=u^{k+2}$.

If $2\leq k < (n-1)/2$ then $vu^2=u^{k+1}$. From
$u^{l+1}=vvu=vu^k=u^{2k-1}$ it follows that $l = 2k-2$ which leads to
$\lceil(n-1)/2\rceil -2$ semigroups with presentations $T_{5,k}$.

If $(n-1)/2\leq k \leq n-5$ then $l \in\{n-3,n-2\}$ which gives
$2(n-\lceil(n-1)/2\rceil-4)$ semigroups with presentations $T_{5,2,k}$
and $T_{5,3,k}$.

If $n-4\leq k\leq n-2$ then $vu^2\in\{u^{n-3},u^{n-2}\}$ and $l
\in\{n-3,n-2\}$, leading to $12$ semigroups with presentations
$T_{5,k,a,b}$ with $a,b\in\{2,3\}$.

{\bf Case 5':} $uv=y$. This case leads to the dual semigroups to
those from Case~5.

Only the transposition of the generators $u$ and $v$ might induce an
isomorphism between two of the semigroups considered above.
That no two semigroups from the same case are isomorphic follows
immediately, but there are two semigroups from different cases which
are isomorphic. These are the semigroups from Case 1 with $k=3$ and
from Case 4 with $l=2$. For all other semigroups the transposition of
$u$ and $v$ either induces an automorphism or an isomorphism to a
semigroup in which $|\langle u\rangle|\neq n-2$.
\end{proof}

Numbers of semigroups of coclass $2$ with minimal generating set of 
size $2$ and order less than $7$ are contained in Table~\ref{tab_computed}
in Section~\ref{sec_enum}. There are no such semigroups with less than
$5$ elements. No further considerations
will be undertaken for semigroups of orders $5$ or $6$ as
semigroups of these orders have long been known~\cite{MS55,Ple67} and
are available in the data library in \textsf{Smallsemi}~\cite{smallsemi}.

From Lemma~\ref{lem_coclass_d} we know certain types of semigroups
with coclass~$2$ and minimal generating set of size~$3$. The remaining types
for this case are given in the next theorem. To simplify the statement
of the theorem we introduce a total ordering $\prec$ on the
semigroups of coclass $1$ from Theorem~\ref{thm_cc1} given by their
order of appearance therein. 
Moreover if a semigroup $S$ allows a presentation with relations $R$
we denote the inverted relations that naturally yield a presentation of 
the dual semigroup by $R^{\perp}$ and the dual semigroup
itself by $S^{\perp}$. Note that a semigroup $S$ from
Theorem~\ref{thm_cc1} is commutative if and only if it is self-dual, that is
$S\cong S^{\perp}$.
 
\begin{theorem}
\label{thm_cc2_d3}
Given $n\in\N$ with $n\geq 6$ define for every two semigroups $V,
W\not\cong H_2$ of order $n-1$ from Theorem~\ref{thm_cc1} with $V \preceq
W$ and with presentations $V=\langle u,v\mid Q\rangle$
respectively $W=\langle u,w\mid R\rangle$ the following
presentation(s):
\begin{enumerate}
\item 
$\langle u,v,w \mid Q, R, vw= u^{k+l-1}, vw= wv\rangle$
if $k+l \leq n-2$;
\item
$\langle u,v,w \mid Q, R, vw= u^{n-i}, wv= u^{n-j}\rangle, 2\leq i\leq j \leq 3$ 
if $k+l \geq n-1$ and $(W\cong W^{\perp}$ or $V\cong W)$;
\item 
$\langle u,v,w \mid Q, R, vw= u^{n-i}, wv= u^{n-j}\rangle, 2\leq i,j \leq 3$
if $V\prec W$ and $W\not\cong W^{\perp}$;
\item
$\langle u,v,w \mid Q, R^{\perp}, vw= u^{n-i}, wv= u^{n-j}\rangle, 
2\leq i,j \leq 3$ if $V\not\cong V^{\perp}$ and $W\not\cong W^{\perp}$,
\end{enumerate}
where $k$ and $l$ are given by the relations
$uv=u^k$ and $uw=u^l$ in $Q$ respectively $R$.

The presentations obtained in this way together with the presentations from
Lemma~\ref{lem_coclass_d} for $c=n-3$ and $r=2$ form a complete list 
up to (anti-)isomorphism of representatives of nilpotent semigroups of order
$n$ and coclass $2$ whose minimal generating set has size $3$.

A semigroup from the above list is self-dual if and only if it is commutative
or is from (iv) with $V\cong W$.
\end{theorem}
\begin{proof}
Let $S=\langle u,v,w\rangle$ be a semigroup of order $n, n\geq 6$, coclass
$2$. Due to Corollary~\ref{coro_bound} we may assume without loss of
generality that $|\langle u\rangle|=n-2$. 

If one of $v^2$ or $w^2$ equals $u^2$ then the conditions for
Lemma~\ref{lem_coclass_d} are satisfied and the presentations follow
from there. It remains to consider the case when neither
$v^2$ nor $w^2$ equal $u^2$. The proof follows a similar approach as
the one of Lemma~\ref{lem_coclass_d}. All eight products
$uv,vu,v^2,uw,wu,w^2,vw,$ and $wv$ of two generators not both equal to
$u$ are in the set $S^2 = \{u^k \mid 2\leq k \leq n-2\}$, and knowing
them uniquely determines $S$. The different choices are discussed
below. For each of the resulting multiplications the value of a
product with three factors will depend only on the number of times $v$
and $w$ appear, making all multiplications associative. The only
non-trivial permutation of generators possibly inducing an (anti-)isomorphism
between different multiplications is the transposition of $v$ and $w$.

Denote $V=\langle u,v\rangle$ and $W=\langle u,w\rangle$. Both $V$ and
$W$ are semigroups with $n-1$ elements and coclass $1$. The possible
choices for $V$ and $W$ up to (anti-)isomorphism are the semigroups listed
in Theorem~\ref{thm_cc1} except $H_2$ (taking into
account that $v^2,w^2 \neq u^2$), and the equations $uv=u^k$ and
$uw=u^l$ hold for some $3\leq k,l\leq n-2$. Without loss of generality
we may assume that $V$ is isomorphic to a semigroup from Theorem~\ref{thm_cc1}
and that $V\preceq W$. The latter avoids considering isomorphic semigroups
under the transposition of $v$ and $w$, except if $V$ and $W$ are of the same type.

{\bf Case 1:} $k+l\leq n-2$. From $uvw=uwv=u^{k+l-1}$ it follows that
$vw=wv=u^{k+l-2}$. Using the latter equation as relation together with the
relations of $V$ and $W$ yields the relations for a presentation of
$S$.

{\bf Case 2:} $k+l\geq n-1$. From $uvw=uwv=u^{n-2}$ it follows that
$vw,wv\in\{u^{n-3},u^{n-2}\}$. If $W$ and hence $V$ are commutative 
then the two choices with $vw\neq wv$ lead to a pair of anti-isomorphic
semigroups; and if $V\cong W$ then these two choices lead to a pair of isomorphic
semigroups (under the transposition of $v$ and $w$). Hence in both cases
there are three choices up to (anti-)isomorphism with presentations as given
in (ii). For all further considerations we have $V\prec W$. If in addition $W$
is isomorphic to a semigroup from Theorem~\ref{thm_cc1}, but not commutative then
all four choices for $vw,wv\in\{u^{n-3},u^{n-2}\}$ yield not (anti-)isomorphic
semigroups with presentations as given in (iii). Finally $W$ can be non-commutative
and anti-isomorphic to  a semigroup from Theorem~\ref{thm_cc1}, in other words
isomorphic to $N_1^{\perp}$ or $N_2^{\perp}$. For commutative $V$ this yields semigroups
that are anti-isomorphic to those in (iii). If $V$ is non-commutative then all
four choices for $vw$ and $wv$ yield not (anti-)isomorphic
semigroups with presentations as given in (iv).

The semigroup $S$ is self-dual if and only if it is commutative or the
transposition of $v$ and $w$ induces an isomorphism from $S$ to $S^{\perp}$.
In the latter case $V\cong W^{\perp}$ is required. For commutative $V$ this
yields $V\cong W$ and hence a presentation from (ii). Following the considerations
from above these semigroups are not self-dual. This leaves semigroups from (iv)
where it is easy to verify that the condition $V\cong W^{\perp}$ is sufficient
for a semigroups to be self-dual.
\end{proof}

Numbers of semigroups of coclass $2$ with minimal generating set of 
size $3$ and order less than $6$ are contained in Table~\ref{tab_computed}
in Section~\ref{sec_enum}. The only such semigroup with at most $4$
elements is the zero semigroup of order $4$. No further considerations
will be undertaken for semigroups of order $5$ as
semigroups of this order have long been known~\cite{MS55} and are
available in the data library \textsf{Smallsemi}~\cite{smallsemi}.

The strategy in the proof of the previous theorem can be extended to
inductively determine semigroups of coclass $r$ and generating set of
size $r+1$, the maximal possible. Let $S$ be a semigroup of class $c$ and coclass
$r$ with minimal generating set $\langle
u_1,\dots,u_k,v_{k+1},\dots,v_{r+1}\rangle$ where 
each $u_i, 1\leq i\leq k$ generates a semigroup of class $c$ but none
of the $v_i, k+1\leq i\leq r+1$ does. If we define $V=\langle
u_1,\dots,u_k,v_{k+1},\dots,v_{r}\rangle$ and $W=\langle
u_1,\dots,u_k,v_{r+1}\rangle$ then $V$ would be
known by induction hypothesis and $W$ is one of the semigroups from 
Lemma~\ref{lem_coclass_d}. Hence every such $S$ can be constructed
from a known semigroup with one fewer element. The possible values for
the products
\begin{equation}
\label{eq_products}
v_iv_{r+1}\mbox{ \ and \ }v_{r+1}v_i\mbox{ \ for \ }k+1\leq i\leq r
\end{equation}
 in a
semigroup constructed from $V$ and $W$ are either fixed by
associativity or equal to one of $u_1^c$ and $u_1^{c+1}$. It remains
to avoid (anti-)isomorphic copies of the same semigroup. A
first step towards an orderly algorithm (see~\cite{Rea78}) for the
construction is to define an ordering $\prec$ on the semigroups
listed in Lemma~\ref{lem_coclass_d} and require
\[
\langle u_1,\dots,u_k,v_{i}\rangle \preceq \langle
u_1,\dots,u_k,v_{j}\rangle \mbox{ \ for all \ }k+1\leq i\leq j\leq r+1. 
\]
If $W$ is not (anti-)isomorphic to any of the subsemigroups of $V$ then the
group of (anti-)auto\-mor\-phisms of $V$ determines which choices for the
products in \eqref{eq_products} lead to (anti-)isomorphic
semigroups. Depending on the size of the group the corresponding orbit
calculations might be hard. The situation becomes even more difficult if
$W$ is of the same type like $\langle u_1,\dots,u_k,v_{r}\rangle$ as
permutations leading to isomorphisms to another
constructed semigroup may move $v_{r+1}$. These problems prevent for
the moment that an algorithm usable in practice can be derived from
this inductive approach. 

To classify all semigroups with coclass $3$ the methods presented in
Section~\ref{sec_cc1} and in this section will
need to be extended. For a semigroup $S$ of coclass $3$ it is not
necessarily true that $S^3\setminus S^4$ contains only one element, while
the proofs of Lemma~\ref{lem_coclass_d} and
Theorems~\ref{thm_cc2_d2} and~\ref{thm_cc2_d3} rely on this as a key fact.

\section{Enumeration}
\label{sec_enum}
\noindent
We present formulae for the numbers of semigroups of coclasses $1$ or
$2$ up to (anti-)isomorphism and up to isomorphism,
and formulae for the numbers of commutative semigroups of coclasses $1$ 
or $2$ up to isomorphism. The results are obtained by counting the presentations
in the respective theorems in the previous sections. For small orders the
formulae have been verified computationally and the computed numbers are
presented.

The enumeration for coclass $1$ follows from Theorem~\ref{thm_cc1}.
\begin{corollary}
\label{coro_cc1}
For $n\in\N$ with $n \geq 5$ the number of nilpotent semigroups of
order $n$ and coclass $1$ \dots
\begin{enumerate}
\item \dots counting up to (anti-)isomorphism equals $n+\lfloor n/2 \rfloor$.
\item \dots counting up to isomorphism equals $n+\lfloor n/2 \rfloor+2$.
\item \dots counting commutative semigroups up to isomorphism equals 
$n+\lfloor n/2 \rfloor-2$.
\end{enumerate}
\end{corollary}

For coclass $2$ we shall determine the formulae depending on the size
of the minimal generating set. If the set has size $2$ we need to count the
presentations in Theorem~\ref{thm_cc2_d2}.

\begin{corollary}
For $n\in\N$ with $n \geq 7$ the number of nilpotent semigroups of order $n$,
coclass $2$ and minimal generating set of size $2$ \dots
\begin{enumerate}
\item \dots counting up to (anti-)isomorphism equals
$5n+\lfloor n/2\rfloor -\lceil n/3\rceil -1$.
\item \dots counting up to isomorphism equals
$7n -\lceil n/3\rceil +5$.
\item \dots counting commutative semigroups up to isomorphism equals
$3n+2\lfloor n/2\rfloor -\lceil n/3\rceil -8$.
\end{enumerate}
\end{corollary}

\begin{proof}
For each case considered in the proof of Theorem~\ref{thm_cc2_d2} also the 
number of semigroups has been given. It remains to note that the presentations
that yield a commutative semigroup
either have $uv=vu$ as relation or are $T_{4,i,j,k}$ with $i=j$.
\end{proof}

The enumeration of semigroups of coclass $2$ with minimal generating set of
size $3$ follows from Lemma~\ref{lem_coclass_d} and Theorem~\ref{thm_cc2_d3}.
\begin{lemma}
For $n\in\N$ with $n \geq 6$ the number of nilpotent semigroups of
order $n$, coclass $2$ and minimal generating set of size $3$ \dots
\begin{enumerate}
\item \dots counting up to (anti-)isomorphism equals
\begin{eqnarray*}
&\frac{1}{8}(21n^2+22n-96) &\textrm{ if $n$ is even, and}\\
&\frac{1}{8}(21n^2+36n-81) &\textrm{ if $n$ is odd.}
\end{eqnarray*}
\item \dots counting up to isomorphism equals
\begin{eqnarray*}
&\frac{1}{8}(27n^2+94n-280) &\textrm{ if $n$ is even, and}\\
&\frac{1}{8}(27n^2+112n-243) &\textrm{ if $n$ is odd.}
\end{eqnarray*}
\item \dots counting commutative semigroups up to isomorphism equals
\begin{eqnarray*}
&\frac{1}{8}(15n^2-58n+24) &\textrm{ if $n$ is even, and}\\
&\frac{1}{8}(15n^2-48n+9) &\textrm{ if $n$ is odd.}
\end{eqnarray*}
\end{enumerate}
\end{lemma}

\begin{proof}
We shall first count those semigroups that arise from the presentations
in Lemma~\ref{lem_coclass_d}. We have $r=2$ and hence $c=n-3$ which leads
to $n-5$ presentations of type $\mathcal{H}_k$, a joint $n-4-\lfloor n/2\rfloor$
presentations of type $\mathcal{J}_k$ or $\mathcal{X}$, and $14$ of type
$\mathcal{N}_{k,l,m}^e$. 

To count the presentations listed in Theorem~\ref{thm_cc2_d3} we calculate
that there is a total of 
\[
\sum_{i=1}^{n-1+\lfloor \frac{n-1}{2} \rfloor-1}i = \frac{1}{2}\left(
  \left(n + \left\lfloor \frac{n-1}{2} \right\rfloor \right)^2 - 3
  \left(n + \left\lfloor\frac{n-1}{2} \right\rfloor\right) \right)+ 1
\]
choices for $V$ and $W$ with $V\preceq W$. Each pair fulfils only one
of the conditions from (i), (ii) and (iii) in Theorem~\ref{thm_cc2_d3}.
There are
\[
\sum_{k=3}^{\lceil \frac{n-1}{2} \rceil+1} \left(\sum_{l=k}^{n-k-2}1 +
\sum_{l=\lceil\frac{n-1}{2}\rceil}^{n-k-2}1\right) = \sum_{k=3}^{\lceil
  \frac{n-1}{2} \rceil+1} \left(n-2k-1\right) +
\left(n-\left\lceil\frac{n-1}{2}\right\rceil-k-1\right)
\]
choices that lead to the $1$ presentation from (i) and 
\[
\left(n+\lceil n/2\rceil -5 + n+\lceil n/2\rceil -4\right) 
=8n-8\lceil n/2 \rceil -36 
\]
that lead to the $4$ presentations from (iii). All other pairs fulfil the
condition in (ii) and hence lead to those $3$ presentations. The pairs of
non-commutative $V$ and $W$ fulfil the conditions in
(iv) and each one leads to $3$ presentations. This yields 
and additional $4\cdot 3=12$ presentations from (iv).

Simplifying the sum of all presentations separately for even and
odd integer yields the stated formulae up to (anti-)isomorphism.

The remaining formulae are obtained in a similar way and the proof
is left out.
\end{proof}

\begin{corollary}
For $n \geq 7$ the number of nilpotent semigroups of order $n$,
coclass $2$ \dots
\begin{enumerate}
\item \dots counting up to (anti-)isomorphism equals
\begin{eqnarray*}
&\frac{1}{8}(21n^2+66n-104) - \left\lceil \frac{n}{3} \right\rceil 
&\textrm{ if $n$ is even, and}\\
&\frac{1}{8}(21n^2+80n-93) - \left\lceil \frac{n}{3} \right\rceil 
&\textrm{ if $n$ is odd.}
\end{eqnarray*}
\item \dots counting up to isomorphism equals
\begin{eqnarray*}
&\frac{1}{8}(27n^2+150n-240) - \left\lceil \frac{n}{3} \right\rceil 
&\textrm{ if $n$ is even, and}\\
&\frac{1}{8}(27n^2+168n-203) - \left\lceil \frac{n}{3} \right\rceil 
&\textrm{ if $n$ is odd.}
\end{eqnarray*}
\item \dots counting commutative semigroups up to isomorphism equals
\begin{eqnarray*}
&\frac{1}{8}(15n^2-26n-40) - \left\lceil \frac{n}{3} \right\rceil 
&\textrm{ if $n$ is even, and}\\
&\frac{1}{8}(15n^2-16n-71) - \left\lceil \frac{n}{3} \right\rceil 
&\textrm{ if $n$ is odd.}
\end{eqnarray*}
\end{enumerate}
\end{corollary}

For small orders the semigroups of coclass $1$ or $2$ where determined
using the code from~\cite[Appendix C]{Dis10}. The numbers are given in 
Table~\ref{tab_computed}. They coincide, where applicable, with the
results obtained from the formulae given in this section.

\begin{table}[ht]
\centering
\caption{Numbers of semigroups with coclass $1$ or $2$}
\begin{tabular}{lrrrrrrrrrrr}
\toprule
{\bf type\hfill$\backslash$\hfill order} 
&{\bf 3} &{\bf 4} &{\bf 5} &{\bf 6} &{\bf 7} &{\bf 8} &{\bf 9} 
&{\bf 10} &{\bf 11} &{\bf 12} &{\bf 13}\\
\midrule
\multicolumn{12}{c}{{\it up to (anti-)isomorphism}}\\
\midrule
coclass $1$ 
&1 &8 &7 &9 &10 &12 &13 &15 &16 &18 &19\\
coclass $2$
&0 &1 &84 &142 &184 &218 &288 &328 &412 &460 &557\\
--, 2-generated
&0 &0 &11 &43 &34 &40 &45 &50 &55 &61 &65\\
--, 3-generated
&0 &1 &73 &99 &150 &178 &243 &278 &357 &399 &492\\
\midrule
\multicolumn{12}{c}{{\it up to isomorphism}}\\
\midrule
coclass $1$ 
&1 &9 &9 &11 &12 &14 &15 &17 &18 &20 &21\\
coclass $2$
&0 &1 &118 &219 &284 &333 &434 &491 &610 &677 &813\\
--, 2-generated
&0 &0 &15 &62 &51 &58 &65 &71 &78 &85 &91\\
--, 3-generated
&0 &1 &103 &157 &233 &275 &369 &420 &532 &592 &722\\
\midrule
\multicolumn{12}{c}{{\it commutative up to isomorphism}}\\
\midrule
coclass $1$ 
&1 &5 &5 &7 &8 &10 &11 &13 &14 &16 &17\\
coclass $2$
&0 &1 &23 &42 &67 &86 &123 &146 &193 &222 &278\\
--, 2-generated
&0 &0 &4 &15 &16 &21 &24 &28 &31 &36 &38\\
--, 3-generated
&0 &1 &19 &27 &51 &65 &99 &118 &162 &186 &240\\
\bottomrule
\end{tabular}
\label{tab_computed}
\end{table}

\def\cprime{$'$}

\end{document}